\documentclass[preprint]{imsart}
\advance\topmargin by -\headsep \hbadness5000 \emergencystretch=1cm
\everydisplay{\mathsurround 0pt} 

\RequirePackage[OT1]{fontenc}
\RequirePackage{amsthm,amsmath}
\RequirePackage[numbers]{natbib}
\RequirePackage[colorlinks,citecolor=blue,urlcolor=blue]{hyperref}

\RequirePackage{array}
\RequirePackage{amssymb}
\RequirePackage{amsmath}
\RequirePackage{amsthm}
\RequirePackage{amsmath}
\RequirePackage{enumerate}
\RequirePackage{subfigure}
\RequirePackage{graphics}

\newtheorem{theorem}{Theorem}
\newtheorem{proposition}{Proposition}
\newtheorem{remark}{Remark}

\begin{document}

\begin{frontmatter}
\title{The value at the mode in multivariate $t$ distributions: a curiosity or not?}
\runtitle{The value at the mode in multivariate $t$ distributions: a curiosity or not?}

\begin{aug}
\author{\fnms{Christophe} \snm{Ley}\ead[label=e1]{chrisley@ulb.ac.be}}
\and
\author{\fnms{Anouk} \snm{Neven}\ead[label=e2]{anouk.neven@uni.lu}}

\address{Universit\'e Libre de Bruxelles, ECARES and D\'epartement de
Math\'ematique, Campus Plaine, boulevard du Triomphe, CP210, \\B-1050, Brussels, Belgium\\
Universit\'e du Luxembourg, UR en math\'ematiques, Campus Kirchberg, 6, rue Richard Coudenhove-Kalergi, \\L-1359, Luxembourg, Grand-Duchy of Luxembourg\\
\printead{e1,e2}}

\runauthor{C. Ley and A. Neven}

\affiliation{Universit\'e Libre de Bruxelles and Universit\'e du Luxembourg}

\end{aug}

\begin{abstract}
It is a well-known fact that multivariate Student $t$ distributions converge to multivariate Gaussian distributions as the number of degrees of freedom $\nu$ tends to infinity, irrespective of the dimension $k\geq1$. In particular, the Student's value at the mode (that is, the normalizing constant obtained by evaluating the density at the center) $c_{\nu,k}=\frac{\Gamma(\frac{\nu+k}{2})}{(\pi \nu)^{k/2} \Gamma( \frac{\nu}{2})}$ converges towards the Gaussian value at the mode $c_k=\frac{1}{(2\pi)^{k/2}}$. In this note, we prove a curious fact: $c_{\nu,k}$ tends monotonically to $c_k$ for each $k$, but the monotonicity changes from increasing in dimension $k=1$ to decreasing in dimensions $k\geq3$ whilst being constant in dimension $k=2$. A brief discussion raises the question whether this \emph{a priori} curious finding is a curiosity, \emph{in fine}.
\end{abstract}


\end{frontmatter}

\section{Foreword.}

One of the first things we learn about Student $t$ distributions is the fact that, when the degrees of freedom $\nu$ tend to infinity, we retrieve the Gaussian distribution, which  has the lightest tails in the Student family of distributions. It is also well-known that the mode of both distributions lies at their center of symmetry, entailing that the value at the mode simply coincides with the corresponding Student $t$ and Gaussian normalizing constants, respectively, and that this value $c_{\nu,k}$ of a $k$-dimensional Student $t$ distribution with $\nu$ degrees of freedom tends to $c_k:=\frac{1}{(2\pi)^{k/2}}$, the value at the mode of the $k$-dimensional Gaussian distribution. Now ask yourself the following question: does this convergence of $c_{\nu,k}$ to $c_k$ take place in a monotone way, and would the type of monotonicity (increasing/decreasing) be dimension-dependent? All the statisticians we have asked this question (including ourselves) were expecting monotonicity and nobody could imagine that the dimension $k$ could have an influence on the type of monotonicity, especially because everybody was expecting Gaussian distributions to always have the highest value at the mode, irrespective of the dimension. However, as we shall demonstrate in what follows, this general idea about the absence of effect by the dimension is wrong. Although the Student $t$ distribution is well-studied in the literature, this convergence has nowhere been established to the best of the authors' knowledge. This, combined with the recurrent astonishment when speaking about the dimension-dependent monotonicity, has led us to writing the present note.

\section{Introduction.}

Though already introduced by Helmert (1875), L\"uroth (1876) and Pearson (1895), the (univariate) $t$ distribution is usually attributed to William Sealy Gosset who, under the pseudonym \textit{Student}  in 1908 (see Student 1908), (re-)defined this probability distribution, whence the commonly used expression Student $t$ distribution. This terminology has been coined by Sir Ronald A. Fisher in Fisher (1925), a paper that has very much contributed to making the $t$ distribution well-known. This early success has motivated researchers to generalize the $t$ distribution to higher dimensions; the resulting multivariate $t$ distribution has been studied, \emph{inter alia}, by Cornish (1954) and Dunnett and Sobel (1954). The success story of the Student $t$ distribution yet went on, and nowadays it is one of the most commonly used absolutely continuous  distributions in statistics and probability. It arises in many situations, including \mbox{e.g.} the Bayesian analysis, estimation, hypotheses testing and modeling of financial data. For a review on the numerous theoretical results and statistical aspects, we refer to Johnson \emph{et al.}~(1994) for the one-dimensional and to Johnson and Kotz (1972) for the multi-dimensional setup.

Under its most common form, the $k$-dimensional $t$ distribution admits the density  
$$f_{\nu}^k(\mathbf{x}):=c_{\nu,k} \left( 1+ \|\mathbf{x}\|^2/\nu\right)^{-\frac{\nu+k}{2}}, \quad \mathbf{x} \in \mathbb{R}^k,$$
with   tail weight parameter $\nu\in\mathbb{R}_0^+$ and  normalizing constant  $$c_{\nu,k}=\frac{\Gamma(\frac{\nu+k}{2})}{(\pi \nu)^{k/2} \Gamma( \frac{\nu}{2})},$$
where the Gamma function is  defined by $\Gamma(z)=\int_{0}^\infty \exp(-t)t^{z-1}\, \mathrm{d}t$. The kurtosis of the $t$ distribution is of course regulated by the parameter $\nu$: the smaller $\nu$, the heavier the tails. For instance, for $\nu=1$, we retrieve the fat-tailed Cauchy distribution. As already mentioned, of particular interest is the limiting case when $\nu$ tends to infinity, which yields the multivariate Gaussian distribution with density
$$
(2\pi)^{-k/2} \exp \left( -\frac{1}{2} \|\mathbf{x}\|^2\right), \quad \mathbf{x} \in \mathbb{R}^k.
$$
The $t$ model  thus embeds the Gaussian distribution into a parametric class of fat-tailed distributions. Indeed, basic calculations show that 
$$\lim_{\nu\rightarrow\infty}\left( 1+ \|\mathbf{x}\|^2/\nu\right)^{-\frac{\nu+k}{2}}= \exp \left( -\frac{1}{2} \|\mathbf{x}\|^2\right)$$ 
and $\lim_{\nu\rightarrow\infty}c_{\nu,k}=(2\pi)^{-k/2}=c_k$. It is to be noted that $c_{\nu,k}$ and $c_k$ respectively correspond to the value at the mode (that is, at the origin) of $k$-dimensional Student $t$ and Gaussian distributions.

In the next section, we shall prove that $c_{\nu,k}$ converges monotonically towards $c_k$, in accordance with the general intuition, but, as we shall see, this monotonicity heavily depends on the dimension $k$: in dimension $k=1$, $c_{\nu,k}$ increases to $c_k$, in dimension $k=2$ we have that $c_{\nu,2}=c_2$ while for $k\geq3$ $c_{\nu,k}$ decreases towards $c_k$. Stated otherwise, the probability mass around the center increases with $\nu$ in the one-dimensional case whereas it decreases in higher dimensions, an \emph{a priori} unexpected fact in view of the interpretation of $\nu$ as tail-weight parameter. It is all the more surprising as the variance of each marginal Student $t$ distribution equals $\nu/(\nu-2)$ for $\nu>2$ and any dimension $k$, which is strictly larger than 1, the variance of the Gaussian marginals. An attempt for an explanation hereof is  provided, raising the question whether this is a curiosity or not.


\section{The (curious?) monotonicity result.}\label{main}

Before establishing our monotonicity results, let us start by  introducing some notations that will be useful in the sequel. To avoid repetition, let us mention that the subsequent formulae are all valid for $x\in\mathbb{R}_0^+$. Denoting by $D_x: =D^{(1)}_x$ the first derivative, define
$$\psi(x)=D_{x} \log(\Gamma(x)),$$
the so-called \textit{digamma function} or  \textit{first polygamma function}. The well-known functional equation
\begin{align}\label{Gam}
\Gamma(x+1)=x\,\Gamma(x)
\end{align}
 thus allows to obtain, by taking logarithms and differentiating,
\begin{align}\label{digam}
\psi(x+1)=\psi(x)+\frac{1}{x}.
\end{align}
Another interesting and useful formula is the series representation of the derivatives of $\psi$:
\begin{align}\label{deriv}
D_x^{(n)} \psi(x)=\sum_{j=0}^\infty \frac{(-1)^{n+1} n!}{(x+j)^{n+1}}, \quad n\geq 1.
\end{align}
For on overview and proofs of these results, we refer to Artin (1964). \vspace{0.3cm}

Now, with these notations in hand, we are ready to state the main result of this paper, namely the announced monotonicity result of the normalizing constants $c_{\nu,k}$.

\begin{theorem}[Dimension-based monotonicity of the normalizing constants in $t$ distributions]\label{theo}
For $k \in \mathbb{N}_0$, define the mapping $g_k: \mathbb{R}^+_0\to \mathbb{R}^+$ by $g_k(\nu)=\frac{\Gamma(\frac{\nu+k}{2})}{(\pi \nu)^{k/2} \Gamma( \frac{\nu}{2})}(=c_{\nu,k})$. We have
\begin{enumerate}[(i)]
\item if $k=1$, $g_k(\nu)$ is monotonically increasing in $\nu$;
\item if $k=2$, $g_k(\nu)$ is constant in $\nu$;
\item if $k\geq3$, $g_k(\nu)$ is monotonically decreasing in $\nu$.
\end{enumerate}
\end{theorem}

\begin{proof}[Proof]
(i) Basic calculus manipulations show that 
$$D_\nu (\log g_1(\nu))=\frac{1}{2}\left(\psi \left(\frac{\nu+1}{2}\right)-\psi\left(\frac{\nu}{2}\right)-\frac{1}{\nu}\right),$$
with $\psi$ the digamma function. By (\ref{deriv}), we know that 
$$D_x \psi(x)=\sum_{j=0}^\infty \frac{1}{(x+j)^2} \quad \textnormal{and} \quad D^2_x \psi(x)=-2 \sum_{j=0}^\infty \frac{1}{(x+j)^3}.$$
Thus, $\psi$ is an increasing and concave function on $\mathbb{R}^+_0$.  Using concavity together with identity (\ref{digam}), we have in particular
$$\psi \left(\frac{\nu+1}{2}\right)\geq \frac{1}{2}\psi \left( \frac{\nu}{2} \right) +\frac{1}{2} \psi \left(\frac{\nu}{2}+1\right) = \psi \left(\frac{\nu}{2}\right)+\frac{1}{\nu}.$$
This inequality readily allows us to deduce that $\log g_1(\nu)$, and hence $g_1(\nu)$, is monotonically increasing in $\nu$.

(ii)  If $k=2$, the function $g_2(\nu)$  reduces to $\frac{1}{2 \pi}$ by simply applying (\ref{Gam}), in other words it equals its limit, whence the claim.

(iii) Assume first that $k\geq3$ is even, hence that $k/2$ is an integer. Using iteratively identity (\ref{Gam}), we can write
$$\Gamma\left( \frac{\nu}{2}+\frac{k}{2}\right)= \underbrace{\left( \frac{\nu}{2}+\frac{k}{2}-1\right)\left(\frac{\nu}{2}+\frac{k}{2}-2\right) \cdots \frac{\nu}{2}}_{\frac{k}{2}\textnormal{ factors}} \Gamma\left(\frac{\nu}{2}\right),$$
which implies
$$g_k(\nu)=\frac{\Gamma(\frac{\nu+k}{2})}{(\pi \nu)^{k/2} \Gamma( \frac{\nu}{2})}=\pi^{-k/2} \left(\frac{1}{2}+\frac{k/2-1}{\nu}\right)\left(\frac{1}{2}+\frac{k/2-2}{\nu}\right) \cdots \left(\frac{1}{2}+\frac{1}{\nu}\right)\frac{1}{2}.$$
Since $a+b/\nu$ is monotonically decreasing in $\nu$ when $b>0$, $g_k(\nu)$ happens to be the product of monotonically decreasing  and positive functions in $\nu$. Thus $g_k(\nu)$ is itself monotonically decreasing in $\nu$, which allows to conclude for  $k$ even. 

Now assume that $k\geq3$ is odd. We set $k=2m+1$ with $m \in \mathbb{N}_0$. The proof is based on the same idea as the proof for the one-dimensional case. One easily sees that
\begin{align}\label{odd}D_\nu (\log g_k(\nu))=\frac{1}{2}\left(\psi \left(\frac{\nu+k}{2}\right)-\psi\left(\frac{\nu}{2}\right)-\frac{k}{\nu}\right),\end{align}
with $\psi$ the digamma function. In the rest of this proof, we establish the monotonicity of $g_k(\nu)$ by proving  by induction on $k$ (respectively, on $m$) that $D_\nu (\log g_k(\nu)) \leq 0$ for all $\nu\in\mathbb{R}_0^+$. \vspace{0.3cm}

\noindent \textbf{Base case}: If $m=1$ (which implies $k=3$),  identity (\ref{digam}) yields $\psi(\nu/2)=\psi(\nu/2 +1)-2/\nu$, hence (\ref{odd}) can be rewritten as
\begin{align}\label{3}\frac{1}{2} \left(\psi \left(\frac{\nu+3}{2}\right)- \psi \left(\frac{\nu}{2}+1\right)-\frac{1}{\nu}\right).\end{align}
By concavity of the digamma function, we have the  inequality $$\psi  \left(\frac{\nu}{2}+1\right) \geq \frac{1}{2} \left( \psi \left(\frac{\nu+3}{2}\right)+\psi \left(\frac{\nu+1}{2}\right) \right),$$ and thus (\ref{3}) can be bounded by
$$\frac{1}{2} \left( \frac{1}{2} \left(\psi \left(\frac {\nu+3}{2}\right)-\psi \left( \frac{\nu+1}{2}\right) \right)-\frac{1}{\nu}\right)=\frac{1}{2}\left(\frac{1}{\nu+1}-\frac{1}{\nu}\right)=-\frac{1}{2\nu (\nu+1)}<0,$$
where we have again used (\ref{digam}). So $D_\nu (\log g_3(\nu)) \leq0$ for all $\nu\in\mathbb{R}_0^+$, and the claim holds for the base case.
\vspace{0.3cm}

\noindent \textbf{Induction case}: Assume that the expression in (\ref{odd}) is negative for $k=2m+1$ with $m\in\mathbb{N}_0$. We now show that the claim is true for $k'=2(m+1)+1=k+2$. It follows once more from (\ref{digam}), combined with the fact that $\nu+k\geq \nu$, that
\begin{align*}
\psi \left(\frac{\nu+k'}{2}\right)-\psi\left(\frac{\nu}{2}\right)-\frac{k'}{\nu} =& \,\,\psi \left(\frac{\nu+k}{2}+1\right)-\psi\left(\frac{\nu}{2}\right)-\frac{k}{\nu}-\frac{2}{\nu}\\
=& \,\,\frac{2}{\nu+k}+\psi \left( \frac{\nu+k}{2}\right)-\psi\left(\frac{\nu}{2}\right)-\frac{k}{\nu}-\frac{2}{\nu}\\
\le & \,\,\psi \left( \frac{\nu+k}{2}\right)-\psi\left(\frac{\nu}{2}\right)-\frac{k}{\nu}\\
\le & \,\,0,
\end{align*}
where the final inequality is due to the induction hypothesis. Thus $D_\nu (\log g_k(\nu)) \le 0$ for all odd $k\geq3$, which concludes the proof. 
\end{proof}

\begin{remark}
One easily sees that the induction-based proof also holds under slight modifications for $k\geq4$ even, but we prefer to show this shorter proof for the even-$k$ case. 
\end{remark}

\begin{figure}
\begin{center}
\begin{minipage}{100mm}
\subfigure[]{
\resizebox*{5cm}{!}{\includegraphics{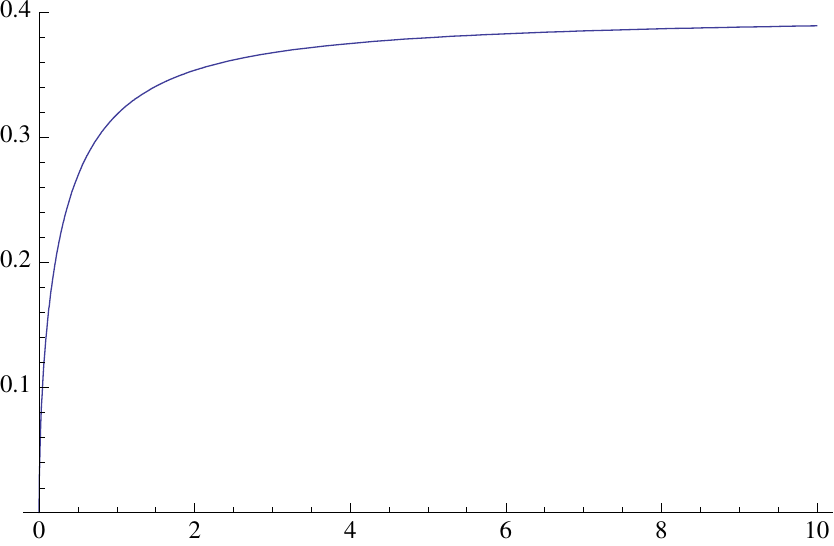}}}%
\subfigure[]{
\resizebox*{5cm}{!}{\includegraphics{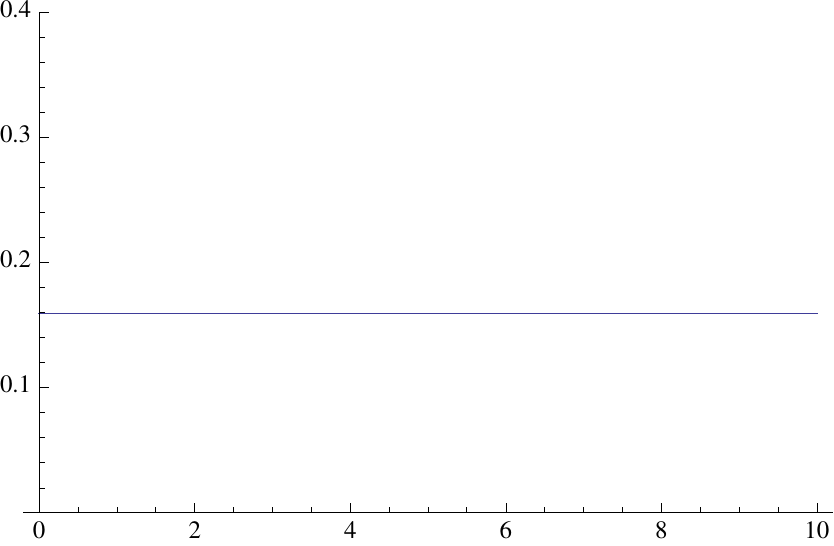}}}\\%
\subfigure[]{
\resizebox*{5cm}{!}{\includegraphics{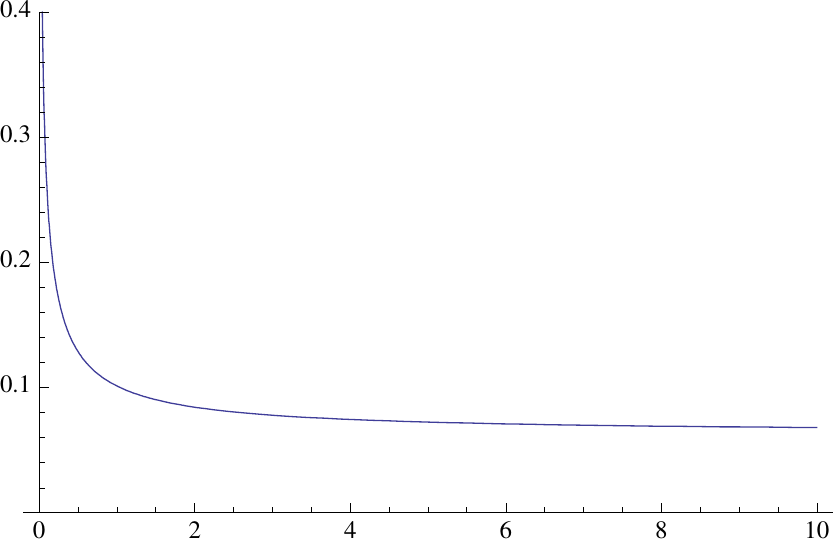}}}%
\subfigure[]{
\resizebox*{5cm}{!}{\includegraphics{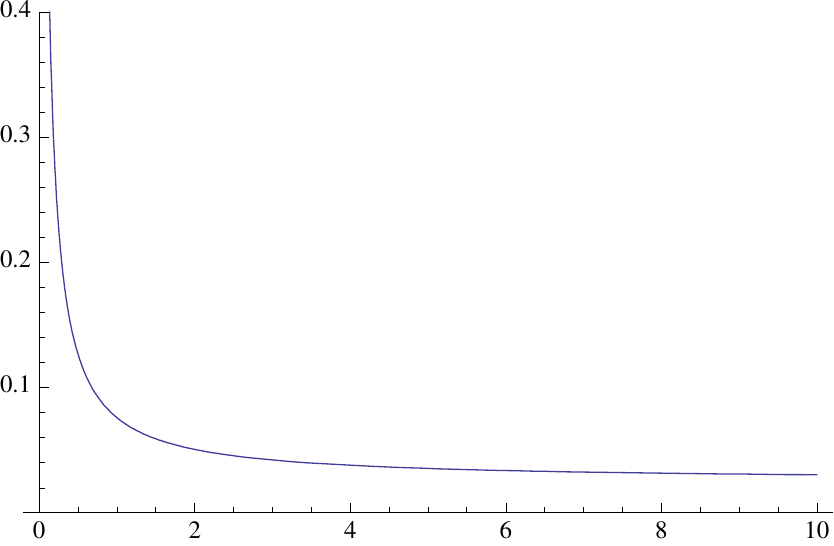}}}\\%
\caption{Plots of $g_k(\nu)=c_{\nu,k}$ for (a) $k=1$, (b) $k=2$, (c) $k=3$, and (d) $k=4$.} \label{constants}
\end{minipage}
\end{center}
\end{figure}

For the sake of illustration, we provide in Figure~\ref{constants} the curves of the values at the mode  $g_k(\nu)=c_{\nu,k}$ for $k=1,2,3$ and $4$. The respective asymptotics of course correspond to the respective limits $(2\pi)^{-k/2}$. While the monotone convergence of the values $c_{\nu,k}$ to $(2\pi)^{-k/2}$ is by no means surprising, the fact that this monotonicity changes from increasing in dimension $k=1$ to decreasing in dimensions $k\geq3$ whilst being constant in dimension $k=2$ seems at first sight, as already mentioned previously, puzzling, as one would expect the Gaussian to always have the highest peak at the center. That this is not the case is illustrated in Figure~\ref{density}, where we have plotted, for distinct values of $k$, several Student $t$ densities with increasing degrees of freedom $\nu$. In dimension $k=1$, the Gaussian density has the highest peak, in dimension $k=2$ all peaks have the same value, whereas for $k\geq3$ the Gaussian has the lowest peak. The same conclusion thus also holds true for the probability mass around the center, which increases with $\nu$ in dimension 1 and decreases for higher dimensions (whilst being nearly unchanged in dimension 2); see Table~\ref{tab}.

\begin{figure}
\begin{center}
\begin{minipage}{140mm}
\subfigure[]{
\resizebox*{7cm}{!}{\includegraphics{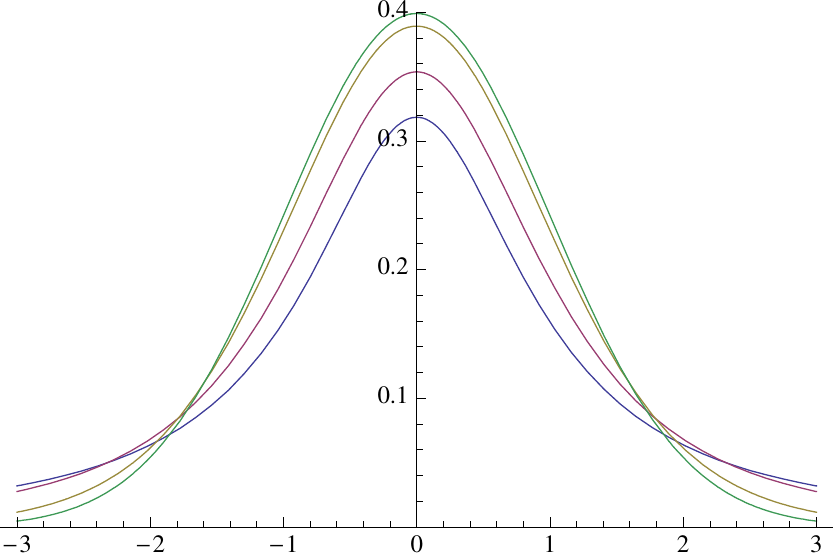}}}%
\subfigure[]{
\resizebox*{7cm}{!}{\includegraphics{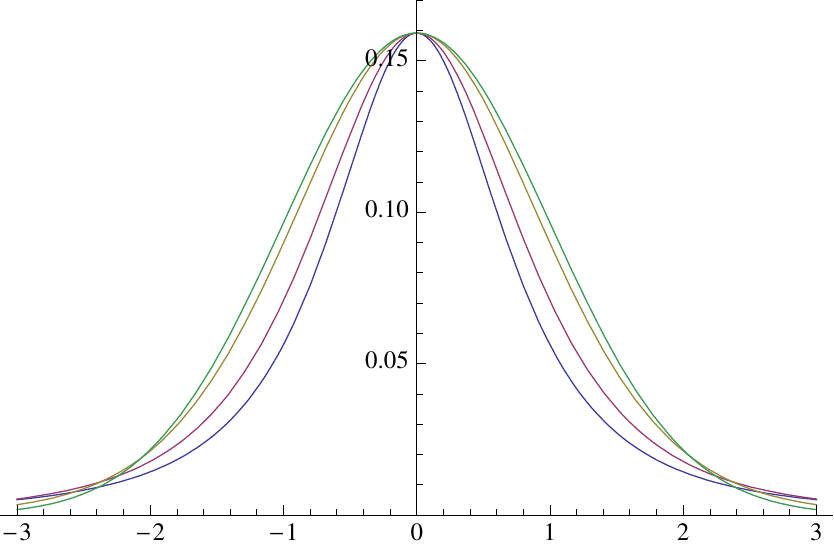}}}\\%
\subfigure[]{
\resizebox*{7cm}{!}{\includegraphics{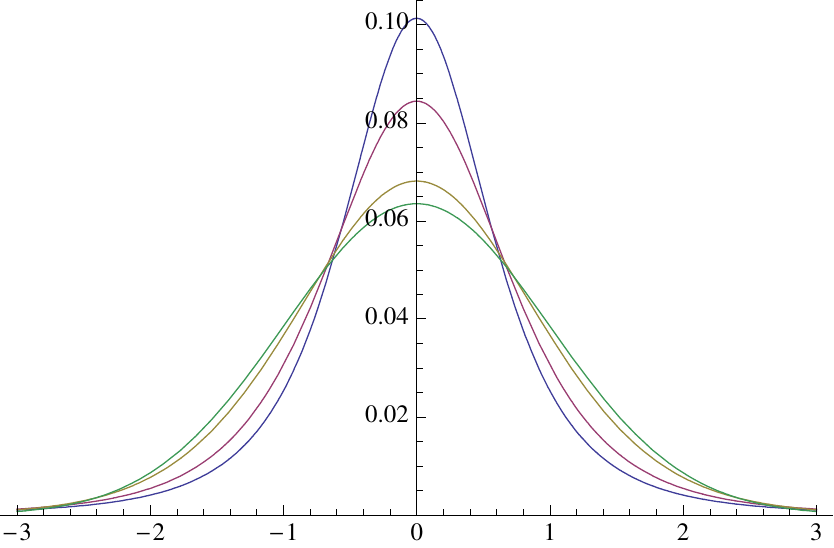}}}%
\subfigure[]{
\resizebox*{7cm}{!}{\includegraphics{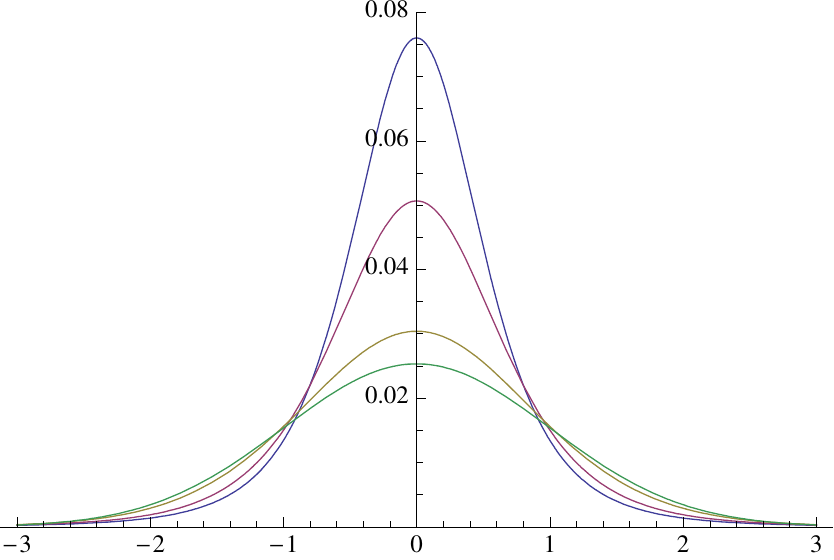}}}\\%
\caption{Plots of the $k$-variate $t$ densities $c_{\nu,k}\left( 1+ \|\mathbf{x}\|^2/\nu\right)^{-\frac{\nu+k}{2}}$ for dimensions (a) $k=1$, (b) $k=2$, (c) $k=3$, and (d) $k=4$. Within each sub-figure, we have chosen four values for the tail parameter $\nu$: $1$ (blue curves), $2$ (red curves), $10$ (yellow curves) and $\infty$ (green curves). From dimension $2$ onwards, the densities are plotted along the first vector of the canonical basis of $\mathbb{R}^k$.} \label{density}
\end{minipage}
\end{center}
\end{figure}

\begin{table} 
\begin{center}
\begin{tabular}{|lc|cccc|}
\hline
&& $k=1$&$k=2$&$k=3$&$k=4$\\\hline
$\nu=1$ & & 0.063451 &0.00496281 &0.000419374& 0.0000368831   \\
$\nu=2$ && 0.070535 &0.00497512 & 0.000350918 &0.0000247519 \\
$\nu=10$&& 0.077679 &0.00498503& 0.000284236& 0.0000149302 \\
$\nu=\infty$&& 0.079656 &0.00498752& 0.000265165& 0.0000124584  \\
\hline
\end{tabular}
\caption{Probabilities that  ${\pmb X}\sim f_\nu^k$  lies inside the $k$-dimensional ball with radius 0.1, for four distinct values of the dimension $k$ and four distinct values of the number of degrees of freedom $\nu$, with $\nu=\infty$ standing for the Gaussian distribution.} \label{tab}
\end{center}
\end{table}

Why does this result seem so counter-intuitive? One reason is that, since the kurtosis of Student $t$ distributions decreases when $\nu$ increases, one would naturally expect that the probability mass around the center always increases with $\nu$, all the more so as the variance-covariance of a $k$-dimensional Student distribution equals, for $\nu>2$, $\frac{\nu}{\nu-2}I_k$, with $I_k$ the $k$-dimensional identity matrix, hence the marginals' variance is always larger as that of the Gaussian marginals. 

This result is even more astonishing as the moment-based kurtosis ratio between two $t$ distributions does not alter with the dimension. Indeed, letting, without loss of generality, $\mathbf{X}_1$ and $\mathbf{X}_2$ be  $k$-variate random vectors following each a $t$ distribution with respective parameters $\nu_1$ and $\nu_2$, straightforward calculations show that, for $m<\min(\nu_1,\nu_2)$,
\begin{align*}
\frac{{\rm E}[||\mathbf{X}_1||^m]}{{\rm E}[||\mathbf{X}_2||^m]}& = \frac{\int_0^\infty r^{m+k-1}c_{\nu_1,k}( 1+ r^2/\nu_1)^{-\frac{\nu_1+k}{2}}dr}{\int_0^\infty r^{m+k-1}c_{\nu_2,k}(1+r^2/\nu_2)^{-\frac{\nu_2+k}{2}}dr}\\
&= \frac{\frac{\nu_1^{m/2}\Gamma(\frac{k+m}{2})\Gamma(\frac{\nu_1-m}{2})}{2\pi^{k/2}\Gamma(\frac{\nu_1}{2})}}{\frac{\nu_2^{m/2}\Gamma(\frac{k+m}{2})\Gamma(\frac{\nu_2-m}{2})}{2\pi^{k/2}\Gamma(\frac{\nu_2}{2})}}\\
&=\left(\frac{\nu_1}{\nu_2}\right)^{m/2}\frac{\Gamma(\frac{\nu_1-m}{2})\Gamma(\frac{\nu_2}{2})}{\Gamma(\frac{\nu_2-m}{2})\Gamma(\frac{\nu_1}{2})},
\end{align*}
which does not depend on the dimension $k$. This result can be summarized in the following proposition.
\begin{proposition}\label{prop}
Let $\mathbf{X}_1$ and $\mathbf{X}_2$ be  $k$-variate random vectors following each a $t$ distribution with respective parameters $\nu_1$ and $\nu_2$. Then, for all $m<\min(\nu_1,\nu_2)$, the ratio $\frac{{\rm E}[||\mathbf{X}_1||^m]}{{\rm E}[||\mathbf{X}_2||^m]}$ does not depend on the dimension $k$. Consequently, the moment-based kurtosis ratio or fourth standardized moment ratio 
$$\frac{\beta_{\nu_1,k}}{\beta_{\nu_2,k}}:=\frac{\frac{{\rm E}[||\mathbf{X}_1||^4]}{({\rm E}[||\mathbf{X}_1||^2])^2}}{\frac{{\rm E}[||\mathbf{X}_2||^4]}{({\rm E}[||\mathbf{X}_2||^2])^2}}$$
is the same for each dimension $k$, provided that $\min(\nu_1,\nu_2)>4$.
\end{proposition}

This proposition seems to add further confusion about the result of Theorem~\ref{theo}. Why is our intuition so misleading? The reason lies most probably  in our general understanding of heavy tails in high dimensions. While, in dimension 1, one can clearly observe the tail-weight by looking at the density curves far from the origin, this visualization vanishes more and more with the dimension. This waning difference is however thwarted by the increase in dimension, in the sense that the smaller difference in height between the density curves is integrated over a larger domain, which explains for instance  the kurtosis ratio result of Proposition~\ref{prop}. It also explains why, in higher dimensions, the lighter-tailed distributions need not have the highest peaks at the mode, as this difference in peak has weak importance in view of the small domain over which it is integrated compared to the tails.

The dimension thus has an impact on the monotone convergence of the values at the mode $c_{\nu,k}$ of multivariate Student $t$ distributions towards the value at the mode $c_k$ of multivariate Gaussian distributions. Curiosity or not?

\vspace{0.5cm}

\noindent ACKNOWLEDGEMENTS: \vspace{0.2cm}

\noindent The research of Christophe Ley is supported by a Mandat de Charg\'e de Recherche du Fonds National de la Recherche Scientifique, Communaut\'e fran\c caise de Belgique. The research of Anouk Neven is supported 
 by an AFR grant of the Fonds National de la Recherche, Luxembourg (Project Reference 4086487). Both authors thank Davy Paindaveine for interesting discussions on this special topic.


\begin{thebibliography}{99}

\bibitem{A64} Artin, E. (1964). \emph{The Gamma function}. New York: Holt, Rinehart and Winston.

\bibitem{C54} Cornish, E. A. (1954). The multivariate $t$-distribution associated with a set of normal sample deviates. \emph{Aust. J. Phys.} {\bf 7}, 531--542.

\bibitem{DS54} Dunnett, C. W., and Sobel, M. (1954). A bivariate generalization of Student's $t$-distribution, with tables for certain special cases. \emph{Biometrika} {\bf 41}, 153--169.

\bibitem{F25} Fisher, R. A. (1925). Applications of ``Student's'' distribution. \emph{Metron} \textbf{5}, 90--104.

\bibitem{H1875} Helmert, F. R. (1875). \"Uber die Bestimmung des wahrscheinlichen Fehlers aus einer endlichen Anzahl wahrer Beobachtungsfehler. \emph{Z. Math. Phys.} {\bf 20}, 300--303.

\bibitem{JK72} Johnson, N. L., and Kotz, S. (1972). \emph{Distributions in Statistics: Continuous multivariate distributions}. New York: Wiley.  

\bibitem{JKB95} Johnson, N. L., Kotz, S., and Balakrishnan, N. (1994). \emph{Continuous univariate distributions}, vol. 2. New York: Wiley, 2nd edition.

\bibitem{L1876} L\"uroth, J. (1876). Vergleichung von zwei Werten des wahrscheinlichen Fehlers. \emph{Astron. Nachr.} {\bf 87}, 209--220.

\bibitem{P1895} Pearson, K. (1895). Contributions to the mathematical theory of evolution, II. Skew variation in homogeneous material. \emph{Philos. T. Roy. Soc. Lond. A} {\bf 186}, 343--414.

\bibitem{Student} Student (1908). The probable error of a mean. \emph{Biometrika} {\bf 6}, 1--25.


\end{thebibliography}
\end{document}